\documentclass[12pt,oneside]{amsart}
\usepackage[russian,english]{babel}
\usepackage{amssymb,amscd,amsmath,latexsym}
\usepackage[all]{xy}
\usepackage[cp1251]{inputenc}
\usepackage{eucal}

\theoremstyle{plain}
\newtheorem{theorem}{Theorem}
\newtheorem{cor}[theorem]{Corollary}
\newtheorem{prop}[theorem]{Proposition}

\theoremstyle{definition}

\newtheorem{example}[theorem]{Example}

\theoremstyle{remark}

\newcommand{\HM}{\mathrm{H\!M}}

\newcommand{\Sm}{\bf Sm/k}

\newcommand \id{\text{id}}

\newcommand \Mor{\text{Mor }}
\newcommand \Ob{\text{Ob }}

\newcommand \spe{\operatorname{Spec}}

\renewcommand \id{\operatorname{id}}

\newcommand \Hom{\operatorname{Hom}}

\newcommand \pt{\operatorname{pt}}

\renewcommand \phi\varphi

\newcommand \Z{\mathbb Z}

\title{{\bf Duality Theorem for Motives}}
\author{Panin, I., Yagunov, S.}

\begin{document}
\maketitle
\pagestyle{plain}

\section*{Introduction}
\English

Recently, due to the active study of cohomological invariants in algebraic geometry,
``transplantation" of classical topological constructions into the algebraic-geometrical ``soil"
seems to be rather important. In particular, it is very interesting to study topological properties of
the category of motives.

The notion of a motive was introduced by Alexander Grothendieck in 1964 in order to formalize notion of the universal
(co-)homology theory (see the detailed exposition of Grothendieck's ideas in~\cite{Ma}).
For us, the principal example of the considered type is a category of motives $\bf{DM^-}$,
constructed by Voevodsky~\cite{Vo2} for algebraic varieties.

Poincar\'e Duality is one of the classical and fundamental results in algebraic topology that first appeared
in his first topological memoir "Analy\-sis Situs"~\cite{Po} (as a part of the statement concerning symmetry of Betti numbers).
The proof of the general duality theorem for extraordinary cohomology theories
apparently belongs to Adams~\cite{Ad}.

The purpose of this paper is to establish the general duality theorem for the category of motives.
Essentially, we extend the main statement of~\cite{PY} to this category.
Many known results can be easily interpreted in these terms. In particular, we get a
generalization of Friedlander--Voevodsky's duality theorem~\cite{FV} to the case of arbitrary base field characteristic.
The proof of this fact, utilizing the main result of~\cite{PY}, was kindly conveyed to the authors by Andrei Suslin~\cite{Su1}.
Being inspired by his work and Dold--Puppe's~\cite{DP} category approach to the duality phenomenon in topology, we decided to write a short,
simple, and self-contained proof of the similar result for the category of motives.

It can be easily seen that our result might be considered as a purely abstract theorem and rewritten
in the spirit of ``abstract nonsense" as a statement about some category with a distinguished class of
morphisms. For the proof one, essentially, uses the existence of finite fiber-products and the terminal object
in the category of varieties, a small part of the tensor triangulated category structure for motives and, finally,
 the existence of transfers for the class of morphisms, generated by graphs of a special type (of projective morphisms).

However, we rather preferred to formulate all the statements for motives of algebraic varieties
in order to clarify geometric nature of the construction and make possible applications easier.
This yielded, in particular, to the appearing of the second index in (co-)homology, responsible for
twist with the Tate object $\Z(1)$ (see Voevodsky~\cite{Vo2}).
  The only exemption is classical Example~\ref{classic}.

\section*{Axioms and Examples}
Consider a contravariant functor $\bf M\colon \Sm\to \mathfrak M$ from the category of smooth algebraic varieties over a field $\mathbf k$
to a tensor triangulated category\footnote{In particular, we assume that the tensor structure is compatible with the shift functor in such a way that
$A[q]\otimes B=A\otimes B[q]$. A full list of axioms for triangulated category can be found in~\cite{May,Maz}.}
$\mathfrak M$, sending cartesian products to tensor products.

For the terminal object $\pt=\spe \rm\bf{k}$ of the category $\Sm$ denote the object $M(\pt)\in\Ob\mathfrak M$ by $\Z$.
From now on, we will often use implicitly the canonical isomorphisms
$$\Z\otimes M(X)\simeq M(X)\simeq M(X)\otimes \Z,$$ induced by the natural identifications $\pt\times X = X = X\times\pt$ in $\Sm$.

Let us also assume that the category $\mathfrak M$ is endowed with a fixed invertible object $\Z(1)$, called Tate object.
We will denote the $(n+1)$-fold product $M\otimes{\Z(1)\otimes\cdots\otimes\Z(1)}$ by $M(n)$.

For a variety $X$ we call an object
$M(X)$ in the category $\mathfrak M$ {\it (orientable) motive} of $X$  and the functor $\bf M$ itself
{\it (orientable) theory of motives on the category $\Sm$}, provided that the axioms below are fulfilled.

\begin{itemize}
  \item {\bf Cancelation axiom.} For every integer $q$ and arbitrary varieties $X$ and $Y$ from $\Sm$ there exists the canonical isomorphism
  $$\Hom_{\mathfrak M}(M(X),M(Y))\simeq \Hom_{\mathfrak M}(M(X)(q),M(Y)(q)).$$
  \item {\bf Transfer axiom.} Every projective equidimensional morphism $f\colon X\to Y$ of codimension $d=\dim Y-\dim X$ determines a map of
  motives $$f^!\colon M(Y)\to M(X)(d)[2d],$$  that is functorial with respect to the specified class of morphisms, {\it i.e.} $f^!(\id)=\id$ and
  $(fg)^!=g^!f^!$.
  \item {\bf Base-change axiom.} For every transversal square (see e.g.~\cite[Definition A.1]{PY}):
\begin{equation*}
\xymatrix{
X^{\prime}\ar^-{\tilde g}[r]\ar_-{\tilde f}[d] & Y^{\prime}\ar^-{f}[d]\\
X\ar^-{g}[r] & Y, }
\end{equation*}
with projective equidimensional morphism $f$ (and, therefore, $\tilde f$),
the relation\footnote{
Here and below $g_*$ {\it et cetera} denote morphisms obtained from the induced morphism $M(g)$ in the category $\mathfrak M$
after application of necessary shift and Tate twist.}
 $\tilde g_*\tilde f^!=f^!g_*$ holds in $\mathfrak M$.
\item {\bf Consistency axiom.} Let $F\colon X\times Z\to Y\times Z$ be a morphism $f\times\id\in\Mor\mathfrak M$ for
some $f\colon X\to Y$ for which the transfer is well-defined. Then, the
 relation $F^!=f^!\otimes 1$ holds in the category $\mathfrak M$ for all such $f,F$.
\end{itemize}

For a given theory of motives $\bf M$ let us define homology and cohomology groups in the following way. Set
$$\HM^{nm}(X):=\Hom_\mathfrak M(M(X),\Z(m)[n])$$
and
$$\HM_{nm}(X):=\Hom_\mathfrak M(\Z(m)[n],M(X)).$$

For us, the most important example of the described construction is Voevodsky's motives.
\begin{example}\label{Voevodsky}
One sets $\mathfrak M=\mathbf{DM^-(k)}$ and the functor $\bf M$ to be the corresponding motive functor from~\cite{SV,Vo2}.
Then $\bf M$ happens to be a functor with transfers for projective morphisms and our list of  axioms is fulfilled.
The Cancelation axiom was checked in this context in~\cite{Vo1}. Construction of the transfer and verification of other axioms
belong to the base properties of Voevodsky's motives (see, for example~\cite[Section 4]{SV}).
\end{example}

\begin{example}\label{classic}
Changing, in our settings, $\Sm$ with a category of  $CW$-complexes and taking the derived
homotopy category of singular complexes with $\Z_2$-coefficients as  $\mathfrak M$,
we obtain, as one can easily see, the standard definitions of singular (co-)homology.
To make the indexes consistent, it suffices to put $M(i)[j]:=M[j-i]$, where on the right-hand-side
the usual triangulated category shift is employed. Alteratively, one can just rewrite all the constructions
with a single index. The role of projective morphisms is played in this case by
proper (in the topological sense) maps. The transfer construction can be found in almost any algebraic topology textbook.
Direct verification shows that all the axioms hold in the considered case. As a result we obtain a proof of
the classical Poincar\'e duality theorem in the spirit of Dold--Puppe's paper~\cite{DP}.

\end{example}

\section*{Statements and Proofs}
The purpose of the current paper is to prove the duality theorem for orientable motives. Namely, we establish the following statement.
\begin{theorem}\label{main}
For every orientable theory of motives  $\bf M$ and arbitrary varieties  $X,Y\in\Sm$ with projective equidimensional $X$ there is the canonical
isomorphism of abelian groups:
$$\Hom(M(Y)(i)[j],M(X))\simeq \Hom(M(Y)\otimes M(X),\Z(d-i)[2d-j]),$$
contravariant with respect to $Y$. Here $d=\dim X$.
\end{theorem}
As a simple consequence of our statement, we get the following variant of the classical Poincar\'e duality theorem.
\begin{cor}\label{Cor}
For an arbitrary smooth projective equidimensional variety $X$ of dimension $d$ there exists the canonical isomorphism:
$$\HM_{*,*}(X)\simeq \HM^{2d-*,d-*}(X).$$
\end{cor}
\begin{proof}
Set $Y=\pt$ in the theorem statement.
\end{proof}
Within the next paragraph we have to assume that the category $\mathfrak M$ under consideration admits
inner $\text{\it \underline{Hom}}$-objects, {\it i.e.} for every variety $X\in\Sm$ the tensor product functor
$M(X)\otimes -\colon \mathfrak M\to\mathfrak M$ has right adjoint.

For a variety $X$, satisfying the hypothesis of Corollary~\ref{Cor},
consider the canonical morphism:
$$\xymatrix@1{M(X)\otimes M(X)\ar@{=}[r]&M(X\times X)\ar^{\Delta^!}[r]&M(X)(d)[2d]\ar^(0.6){p_*}[r]&\Z(d)[2d]}$$
and call it {\it motivic cofundamental class} (see the details below).
This morphism determines the canonical duality map: $$M(X)\to\text{\it \underline{Hom}}(M(X),\Z(d)[2d]).$$
Then the following statement holds.
\begin{cor}\label{Dual} Let $\bf M$ be the theory of motives considered above in Example~\ref{Voevodsky}.
Then for every smooth projective equidimensional variety $X$ the constructed duality morphism
$$\psi\colon M(X)\to\text{\it \underline{Hom}}(M(X),\Z(d)[2d])$$
is an isomorphism.
\end{cor}
\begin{proof}
Complete the morphism $\psi$ up to a distinguished triangle in the
category $\mathfrak M$ and denote the third vertex of the triangle
by $\text{\rm Cone}(\psi)$. As it follows from Theorem~\ref{main},
the relation $\Hom(M(Y)(p)[q],\text{\rm Cone}(\psi))=0$
holds for every smooth irreducible variety $Y$ and arbitrary
integers $p,q$. Because the category $\mathfrak M$ is weakly generated by
motives of such a kind, one can see that the object $\text{\rm Cone}(\psi)$
is isomorphic to the zero-object. Triangulated category axioms
easily imply that the morphism with zero cone is an
isomorphism.
\end{proof}

\noindent {\it Proof of Theorem~\ref{main}.}
In order to construct isomorphisms mentioned in the Theorem, one needs the following ingredients.
First of all, denote by $1\in\HM^{0,0}(\pt)=\HM_{0,0}(\pt)$ the element of
(co-)homology group of the point, corresponding to the identity morphism $\id\colon\pt\to\pt$.
Let us note that for a morphism of varieties $f\colon X\to Y$ the morphism $M(f)$  induces the natural maps
$f_*\colon\HM_{*,*}(X)\to\HM_{*,*}(Y)$ in motivic homology and $f^*\colon\HM^{*,*}(Y)\to\HM^{*,*}(X)$
respectively, in cohomology.
In case then $f$ is projective of codimension $d$, motivic transfer map $f^!$ induces
corresponding transfers
\begin{equation*}
f_!\colon\HM^{*,*}(X)\to\HM^{*+2d,*+d}(Y)\hskip 0.3cm\text{and}\hskip 0.3cm
f^!\colon\HM_{*,*}(Y)\to\HM_{*-2d,*-d}(X)
\end{equation*}
in (co-)homology.
For a variety $X$ of dimension $d$ consider the diagonal and projection morphisms:
$$
X\times X\overset\Delta\leftarrow X\overset p\rightarrow\pt
$$
and call the elements  $$[X]_*=\Delta_*p^!(1)\in\HM_{2d,d}(X\times X)$$
 and $$[X]^*=\Delta_!p^*(1)\in\HM^{2d,d}(X\times X)$$
{\it fundamental} and {\it cofundamental classes} of $X$, correspondingly.
Let us also define slant-products
$$/\colon \HM^{i,j}(X\times Y)\otimes\HM_{m,n}(Y)\to \HM^{i-m,j-n}(X)$$ and
$$\backslash\colon \HM_{m,n}(X\times Y)\otimes\HM^{i,j}(Y)\to \HM_{m-i,n-j}(X)$$
in the following way. For elements $\alpha\in \HM^{i,j}(X\times Y)$ and $a\in \HM_{m,n}(Y)$ set:
$$ \alpha/a: M(X)(m)=M(X)\otimes\Z(m)\overset{1\otimes a}\to M(X)\otimes M(Y)[-n]\overset\alpha\to \Z(i)[j-n]$$
for the first product, and symmetrically:
$$ \beta\backslash b:\Z(m)[n-j]\overset b\to M(X)\otimes M(Y)[-j]\overset{1\otimes \beta}\to M(X)\otimes\Z(i)=M(X)(i)$$
 for the second one, provided that $b\in \HM_{m,n}(X\times Y)$ and $\beta\in \HM^{i,j}(Y)$.
 (Here and below we implicitly use Cancelation axiom.)
Set Poincar\'e Duality homomorphisms for the case of Corollary~\ref{Cor}, letting:
$$\mathcal D_\bullet(-)=[X]^*/-\colon \HM_{*,*}(X)\to \HM^{2d-*,d-*}(X)$$
and
$$\mathcal D^\bullet(-)=-\backslash [X]_*\colon \HM^{*,*}(X)\to \HM_{2d-*,d-*}(X),$$
verbatim as it was done in~\cite{PY}.

Then, consider a natural extension of these homomorphisms to the general case of the theorem.
Denote, for brevity, the motive $M(Y)(i)[j]$ by $\mathcal Y$, and the index shift and twist $(d)[2d]$  by $\{d\}$.
Construct a map
$$\mathcal D_\bullet\colon \Hom(\mathcal Y,-)\to\Hom(\mathcal Y\otimes -,\Z\{d\})$$
as follows.
For a morphism $a\in \Hom(\mathcal Y,M(X))$ give $D_\bullet(a)$ by the formula:
$$\mathcal Y\otimes M(X)\overset{a\otimes 1}\to M(X)\otimes M(X)\overset{\Delta^!}\to M(X)\{d\}\overset{p_*}\to \Z\{d\}.$$
The inverse map
$$\mathcal D^\bullet\colon\Hom(\mathcal Y\otimes -,\Z\{d\})\to \Hom(\mathcal Y,-)$$
 is given for a morphism $\alpha\in \Hom(\mathcal Y\otimes M(X),\Z\{d\})$ as:
{
$$\mathcal Y\overset{1\otimes p^!}\to \mathcal Y\otimes M(X)\{-d\}\overset{1\otimes \Delta_*}\to \mathcal Y\otimes M(X)
\otimes M(X)\{-d\}\overset{\alpha\otimes 1}\to \Z\otimes M(X).$$}
The constructed maps are obviously contravariant with respect to the variable $\mathcal Y$.
Theorem~\ref{main} now results from the Cancelation axiom and commutativity of two diagrams below\footnote{In the consequent
diagrams
$M(X)^{\otimes n}$ denotes $n$-fold tensor product and $\Delta^i$ is a diagonal morphism applied to
the $i$-th factor.}:
{
$$
\xymatrix{\mathcal Y\otimes\Z\otimes M(X)\ar_{1\otimes p^!\otimes 1}[d]\ar^{\id}@{=}[dr]\\
\ar@{}[dr]|{\blacksquare}\mathcal Y\otimes M(X)^{\otimes 2}\{-d\}\ar_{\Delta^3_*}[d]\ar^(0.55){1\otimes \Delta^!}[r] &
 \mathcal Y\otimes M(X)\ar_{\Delta_*}[d]\ar^{\id}@{=}[dr]\\
\mathcal Y\otimes M(X)^{\otimes 3}\{-d\}\ar^(0.55){1\otimes 1\otimes\Delta^!}[r]\ar_{\alpha\otimes1\otimes1}[d] &
\mathcal Y\otimes M(X)^{\otimes 2}\ar^(0.55){1\otimes1\otimes p_*}[r]\ar_{\alpha\otimes1}[d]& \mathcal Y\otimes M(X)\ar_\alpha[d] \\
\Z\otimes M(X)^{\otimes 2}\ar^{1\otimes\Delta^!}[r]& \Z\otimes M(X)\{d\}\ar^(0.65){1\otimes p_*}[r] & \Z\{d\}}$$}
and
{
$$\xymatrix{\mathcal Y\ar^-a[r]\ar_{p^!}[d]& M(X)\otimes\Z\ar_{1\otimes p^!}[d]\ar^{\id}@{=}[dr]\\
\mathcal Y\otimes M(X)\{-d\}\ar_{\Delta_*}[d]\ar^{a\otimes 1}[r] & \ar@{}[dr]|{\blacksquare}M(X)^{\otimes 2}\{-d\}\ar^(0.6){\Delta^!}[r]\ar_{\Delta^2_*}[d]&
M(X)\ar_{\Delta_*}[d]\ar^{\id}@{=}[rd]\\
\mathcal Y\otimes M(X)^{\otimes 2}\{-d\}\ar^(0.55){a\otimes 1\otimes 1}[r]&
M(X)^{\otimes 3}\{-d\}\ar^(0.55){\Delta^!\otimes1}[r]& M(X)^{\otimes 2}\ar^{p_*}[r]& M(X)}.$$}
The squares, marked by $\blacksquare$ signs in both the diagrams,
correspond to the Cartesian transversal diagram of varieties
$$\xymatrix{X\times X\ar_{\Delta^2}[d]&X\ar^\Delta[d]\ar_-\Delta[l]\\
X\times X\times X&X\times X,\ar^-{\Delta^1}[l]}
$$
and, therefore, are commutative due to base-change and consistence axioms.\qed

Finally, let us note that as soon as Poincar\'e duality isomorphisms are been fixed, one can uniquely reconstruct the
transfer maps in\\ {(co-)homology} exactly in the same way as it happens in the classical algebraic topology.
Namely, the following statement holds.
\begin{prop}
For projective equidimensional varieties $X,Y\in\Sm$ and a morphism $f\colon X\to Y$, one has:
\begin{equation*}
f_!=\mathcal D_\bullet^Y f_* \mathcal D^\bullet_X\hskip 1cm\text{and}\hskip 1cm
f^!=\mathcal D^\bullet_X f^* \mathcal D_\bullet^Y.
\end{equation*}
Here $\mathcal D_\bullet^X$ and $\mathcal D_\bullet^Y$ denote Poincar\'e duality isomorphisms
applied to varieties $X$ and $Y$, respectively.
\end{prop}

\begin{proof}
The first equality can be easily derived from the relation
 $$f_*(\alpha\backslash [X]_*)=f_!(\alpha)\backslash [Y]_*,$$ which,
in its row, yields from the commutativity of the diagram:
$$\xymatrix{&&M(X)\otimes M(X)\ar^{1\otimes f_*}[d]\ar^(0.55){\alpha\otimes 1}[r]&\Z\otimes M(X)\ar^{1\otimes f_*}[d]\\
\Z\ar^{p_X^!}[r]\ar_-{p_Y^!}[dr]&M(X)\ar@{}[dr]|{\blacksquare}\ar[r]_-{\Gamma^f_*}\ar^-{\Delta^X_*}[ur]&M(X)\otimes M(Y)\ar[r]^(0.55){\alpha\otimes 1}&\Z\otimes M(Y)\\
&M(Y)\ar_-{\Delta^Y_*}[r]\ar^{f^!}[u]&M(Y)\otimes M(Y).\ar_{f^!\otimes 1}[u]}$$
(Here all the shifts and twists are omitted for simplicity.)
The square marked by $\blacksquare$ is induced by the transversal graph diagram of the morphism $f$:
$$\xymatrix{X\ar^-{\Gamma^f}[r]\ar_f[d]&X\times Y\ar^{f\times\id}[d]\\
Y\ar_-\Delta[r]&Y\times Y}$$
Hence, it is commutative.
Commutativity in all other diagram parts is obvious.
The second equality can be proven in a similar way.
\end{proof}
\Russian

\end{document}